\documentclass[12pt, reqno]{amsart}
\usepackage{amsmath, amsthm, amscd, amsfonts, amssymb, graphicx, color}
\usepackage[bookmarksnumbered, colorlinks, plainpages]{hyperref}
\hypersetup{colorlinks=true,linkcolor=red, anchorcolor=green,
citecolor=cyan, urlcolor=red, filecolor=magenta, pdftoolbar=true}

\newtheorem{theorem}{Theorem}[section]
\newtheorem{lemma}[theorem]{Lemma}
\newtheorem{proposition}[theorem]{Proposition}
\newtheorem{corollary}[theorem]{Corollary}
\theoremstyle{definition}

\theoremstyle{remark}
\newtheorem{remark}[theorem]{Remark}
\numberwithin{equation}{section}

\begin{document}

\setcounter{page}{1}

\title[$A$-numerical radius and product of semi-Hilbertian operators]
{$A$-numerical radius and product of semi-Hilbertian operators}

\author[A. Zamani]
{Ali Zamani}

\address{Department of Mathematics, Farhangian University, Tehran, Iran}
\email{zamani.ali85@yahoo.com}

\subjclass[2010]{Primary 47A05; Secondary 46C05, 47B65, 47A12.}

\keywords{Positive operator, semi-inner product, $A$-numerical radius.}

\begin{abstract}
Let $A$ be a positive bounded operator on a Hilbert space
$\big(\mathcal{H}, \langle \cdot, \cdot\rangle \big)$.
The semi-inner product ${\langle x, y\rangle}_A := \langle Ax, y\rangle$, $x, y\in\mathcal{H},$
induces a seminorm ${\|\cdot\|}_A$ on $\mathcal{H}$.
Let $w_A(T)$ denote the $A$-numerical radius of an operator $T$ in the semi-Hilbertian space
$\big(\mathcal{H}, {\|\cdot\|}_A\big)$. In this paper, for any semi-Hilbertian operators $T$ and $S$,
we show that $w_A(TR) = w_A(SR)$ for all ($A$-rank one) semi-Hilbertian operator $R$
if and only if $A^{1/2}T = \lambda A^{1/2}S$ for some complex unit $\lambda$.
From this result we derive a number of consequences.
\end{abstract} \maketitle
\section{Introduction}
Let $\mathcal{H}$ and $\mathcal{K}$ be complex Hilbert spaces with inner product $\langle\cdot, \cdot\rangle$.
By $\mathbb{B}(\mathcal{H}, \mathcal{K})$ we denote the space of all bounded linear operators from $\mathcal{H}$ to $\mathcal{K}$,
and we abbreviate $\mathbb{B}(\mathcal{H}) = \mathbb{B}(\mathcal{H}, \mathcal{H})$.
For every $T\in\mathbb{B}(\mathcal{H}, \mathcal{K})$ its range is denoted by $\mathcal{R}(T)$,
its null space by $\mathcal{N}(T)$, and its adjoint by $T^*$.
If $\mathcal{M}$ is a linear subspace of $\mathcal{H}$, then $\overline{\mathcal{M}}$ stands for its closure in the
norm topology of $\mathcal{H}$. Given a closed subspace $\mathcal{M}$ of $\mathcal{H}$, $P_{\mathcal{M}}$
denotes the orthogonal projection onto $\mathcal{M}$.
Throughout this paper, we assume that $A\in\mathbb{B}(\mathcal{H})$
is a positive operator (i.e., $\langle Ax, x\rangle \geq 0$ for all $x\in\mathcal{H}$).
Such an $A$ induces a semi-inner product on $\mathcal{H}$ defined by ${\langle x, y\rangle}_A = \langle Ax, y\rangle$
for all $x, y \in\mathcal{H}$. We denote by ${\|\cdot\|}_A$ the seminorm induced by ${\langle \cdot, \cdot\rangle}_A$.
Observe that ${\|x\|}_A = 0$ if and only if $x\in\mathcal{N}(A)$. Then ${\|\cdot\|}_A$ is a norm if and only if $A$ is one-to-one,
and the seminormed space $(\mathcal{H}, {\|\cdot\|}_A)$ is a complete space if and only if $\mathcal{R}(A)$ is closed in $\mathcal{H}$.
For $T\in\mathbb{B}(\mathcal{H})$, the quantity of $A$-operator seminorm of $T$ is defined by
${\|T\|}_A = \sup\big\{{\|Tx\|}_A: \,\,x\in \overline{\mathcal{R}(A)},\, {\|x\|}_A =1\big\}$.
Notice that it may happen that ${\|T\|}_A = + \infty$
for some $T\in\mathbb{B}(\mathcal{H})$.
For example, let $A$ be the diagonal operator on
the Hilbert space $\ell^2$ given by $Ae_n = \frac{e_n}{n!}$,
where $\{e_n\}$ denotes the canonical basis of $\ell^2$
and consider the left shift operator $T \in \mathbb{B}(\ell^2)$.
From now on we will denote
$\mathbb{B}^{A}(\mathcal{H}) :=\big\{T\in \mathbb{B}(\mathcal{H}): \,\, {\|T\|}_A < \infty\big\}$.
For $T\in \mathbb{B}(\mathcal{H})$, an operator $R\in \mathbb{B}(\mathcal{H})$
is called an $A$-adjoint operator of $T$ if for every $x, y\in \mathcal{H}$,
we have ${\langle Tx, y\rangle}_A = {\langle x, Ry\rangle}_A$, that is, $AR = T^*A$.
Generally, the existence of an $A$-adjoint operator is not guaranteed.
The set of all operators that admit $A$-adjoints is denoted by $\mathbb{B}_{A}(\mathcal{H})$.
If $T\in\mathbb{B}_{A}(\mathcal{H})$, then the ``reduced" solution of the equation $AX = T^*A$ is a distinguished $A$-adjoint
operator of $T$, which is denoted by $T^{\sharp_A}$. Notice that if $T\in\mathbb{B}_{A}(\mathcal{H})$,
then $T^{\sharp_A}\in\mathbb{B}_{A}(\mathcal{H})$ and
$(T^{\sharp_A})^{\sharp_A} = P_{\overline{\mathcal{R}(A)}}TP_{\overline{\mathcal{R}(A)}}$.
The set of all operators admitting $A^{1/2}$-adjoints is denoted by $\mathbb{B}_{A^{1/2}}(\mathcal{H})$.
It can be verified that
\begin{align*}
\mathbb{B}_{A^{1/2}}(\mathcal{H}) = \big\{T\in \mathbb{B}(\mathcal{H}): \,\, \exists c>0;
\,\,{\|Tx\|}_{A} < c{\|x\|}_{A}, \,\, \forall x\in \mathcal{H}\big\}.
\end{align*}
Note that $\mathbb{B}_{A}(\mathcal{H})$ and $\mathbb{B}_{A^{1/2}}(\mathcal{H})$ are two subalgebras of
$\mathbb{B}(\mathcal{H})$ which are neither closed nor dense in $\mathbb{B}(\mathcal{H})$.
Moreover, the inclusions
$\mathbb{B}_{A}(\mathcal{H}) \subseteq \mathbb{B}_{A^{1/2}}(\mathcal{H})
\subseteq \mathbb{B}^{A}(\mathcal{H}) \subseteq \mathbb{B}(\mathcal{H})$
hold with equality if $A$ is one-to-one and has a closed range.
For an account of results, we refer the reader to \cite{Ar.Co.Go, Fo.Go}.

The $A$-numerical range of $T\in\mathbb{B}(\mathcal{H})$
is a subset of the set of complex numbers $\mathbb{C}$ and it is defined by
$W_A(T) := \big\{{\langle Tx, x\rangle}_A: \,\, x\in \mathcal{H},\, {\|x\|}_A = 1\big\}$.
It is known as well that $W_A(T)$ is a nonempty convex subset of $\mathbb{C}$ (not necessarily
closed), and its supremum modulus, denoted by $w_A(T) = \sup\big\{|\xi|: \,\, \xi\in W_A(T)\big\}$, is called
the $A$-numerical radius of $T$ (see \cite{Ba.Ka.Ah}).
It is a generalization of the concept of numerical radius of an operator:
recall that the numerical radius of $T\in\mathbb{B}(\mathcal{H})$ is defined as
$w(T) := \sup\big\{|\langle Tx, x\rangle|: \,\, x\in \mathcal{H},\, \|x\| = 1\big\}$.
Notice that it may happen that $w_A(T) = + \infty$ for some $T\in\mathbb{B}(\mathcal{H})$.
For example, consider operators
$A = \begin{bmatrix}
1 & 0 \\
0 & 0
\end{bmatrix}$ and
$T = \begin{bmatrix}
0 & 1 \\
1 & 0
\end{bmatrix}$.
It has recently been shown in \cite[Theorem 2.5]{Z}
that if $T\in\mathbb{B}_{A^{1/2}}(\mathcal{H})$, then
\begin{align}\label{I.S1.2}
w_A(T) = \displaystyle{\sup_{\theta \in \mathbb{R}}}
{\left\|\frac{e^{i\theta}T + (e^{i\theta}T)^{\sharp_A}}{2}\right\|}_A.
\end{align}
In the special case $A = I$, we have
\begin{align}\label{I.S1.3}
w(T) = \displaystyle{\sup_{\theta \in \mathbb{R}}} \big\|\mbox{Re}(e^{i\theta}T)\big\|.
\end{align}
Moreover, it is known that $w_A(\cdot)$ is a seminorm on $\mathbb{B}_{A^{1/2}}(\mathcal{H})$
and satisfies $\frac{1}{2}{\|T\|}_{A} \leq w_A(T)\leq {\|T\|}_{A}$ for all $T\in \mathbb{B}_{A^{1/2}}(\mathcal{H})$.
For proofs and more facts about $A$-numerical radius of operators,
we refer the reader to \cite{Ba.Ka.Ah, Z} and the references therein.
Some other related topics can be found in \cite{Ch.Ka.Ps, K.S, Wi.Se.Su, M.X.Z, Ps, Su, Zh.Ho.He}

If $T, S\in\mathbb{B}_{A^{1/2}}(\mathcal{H})$ and $A^{1/2}T = \lambda A^{1/2}S$ for some complex unit $\lambda$,
then for every $R\in\mathbb{B}_{A^{1/2}}(\mathcal{H})$,
it is easy to see that $|{\langle TRx, x\rangle}_{A}| = |{\langle SRx, x\rangle}_{A}|$ for all $x\in \mathcal{H}$,
and therefore
\begin{align}\label{proplem}
w_{A}(TR) = w_{A}(SR) \qquad \big(R\in\mathbb{B}_{A^{1/2}}(\mathcal{H})\big).
\end{align}
However, the converse implication is not obvious and hence it is natural to
consider the problem of characterizing all operators $T, S\in\mathbb{B}_{A^{1/2}}(\mathcal{H})$
satisfying (\ref{proplem}).
In the next section, by using a construction from \cite{Ar.Co.Go, Br.Ro},
for any $T, S\in\mathbb{B}_{A^{1/2}}(\mathcal{H})$,
we show that $w_A(TR) = w_A(SR)$ (resp., $W_A(TR) = W_A(SR)$) for all $R\in\mathbb{B}_{A^{1/2}}(\mathcal{H})$
if and only if $A^{1/2}T = \lambda A^{1/2}S$ (resp., $A^{1/2}T = A^{1/2}S$) for some complex unit $\lambda$.
\section{Results}
In order to achieve the goals of the present section, we need some prerequisites.
For the positive operator $A\in \mathbb{B}(\mathcal{H})$, the semi-inner product ${\langle \cdot, \cdot\rangle}_{A}$
induces on the quotient $\mathcal{H}/\mathcal{N}(A)$
an inner product which is not complete unless $\mathcal{R}(A)$ is closed.
A canonical construction due to de Branges and Rovnyak \cite{Br.Ro} (see also \cite{Co.Gh, Co.Ma.St})
shows that the completion of $\mathcal{H}/\mathcal{N}(A)$ is isometrically isomorphic
to $\mathcal{R}(A^{1/2})$, with the inner product
\begin{align}\label{I.1}
[A^{1/2}x, A^{1/2}y]: = \big\langle P_{\overline{\mathcal{R}(A^{1/2})}}\,x, P_{\overline{\mathcal{R}(A^{1/2})}}\,y\big\rangle
\qquad (x, y \in \mathcal{H}).
\end{align}
The Hilbert space $\big(\mathcal{R}(A^{1/2}), [\cdot, \cdot]\big)$ will be denoted by $\mathbf{R}(A^{1/2})$.
Moreover, it can be checked that $\mathcal{R}(A)$ is dense in $\mathbf{R}(A^{1/2})$ and by (\ref{I.1}) we have
\begin{align}\label{I.2}
[Ax, Ay] = {\langle x, y\rangle}_{A} \qquad (x, y \in \mathcal{H}),
\end{align}
whence
\begin{align}\label{I.3}
{\|Ax\|}_{\mathbf{R}(A^{1/2})} = {\|x\|}_{A} \qquad (x\in \mathcal{H}).
\end{align}
The following results can be found in \cite{Ar.Co.Go}.
\begin{proposition}\cite[Proposition 2.1]{Ar.Co.Go}\label{L.1}
For operator $M_{A}: \mathcal{H} \rightarrow \mathbf{R}(A^{1/2})$
defined by $M_{A}x = Ax$ the following assertions hold.
\begin{itemize}
\item[(i)] $M_{A} \in \mathbb{B}(\mathcal{H}, \mathbf{R}(A^{1/2}))$
and $\mathcal{R}(M_{A}) = \mathcal{R}(A)$.
\item[(ii)] $M^*_{A}: \mathbf{R}(A^{1/2}) \rightarrow \mathcal{H}$,
$M^*_{A}(A^{1/2}x) = A^{1/2}x$ and $\mathcal{R}(M^*_{A}) = \mathcal{R}(A^{1/2})$.
\end{itemize}
\end{proposition}
\begin{proposition}\cite[Theorem 2.3]{Ar.Co.Go}\label{L.2}
Let $\widetilde{T}: \mathcal{R}(A^{1/2}) \rightarrow \mathcal{R}(A^{1/2})$ and
$Z: \mathcal{R}(A^{1/2}) \rightarrow \mathcal{R}(A^{1/2})$ be linear operators such that
$\big\langle \widetilde{T}(A^{1/2}x), A^{1/2}y\big\rangle = \big\langle A^{1/2}x, Z(A^{1/2}y)\big\rangle$
for every $x, y \in \mathcal{H}$. If $\widetilde{T}$ is bounded in $\mathbf{R}(A^{1/2})$ then $\widetilde{T}$ is bounded in $\mathcal{H}$.
\end{proposition}
\begin{proposition}\cite[Proposition 3.6]{Ar.Co.Go}\label{L.3}
Consider $T\in\mathbb{B}(\mathcal{H})$. Then, there exists $\widetilde{T}\in\mathbb{B}(\mathbf{R}(A^{1/2}))$
such that $\widetilde{T}M_{A} = M_{A}T$ if and only if $T\in\mathbb{B}_{A^{1/2}}(\mathcal{H})$.
In such case $\widetilde{T}$ is unique.
\end{proposition}
Let us recall that by \cite[Lemma 2.4]{Zh.Ho.He} we have
\begin{align}\label{I.4}
w(x \otimes y) = \frac{|\langle x, y\rangle| + \|x\|\,\|y\|}{2},
\end{align}
for all $x, y \in\mathcal{H}$.
Here, $x \otimes y$ denotes the rank one operator in $\mathbb{B}(\mathcal{H})$ defined by
$(x \otimes y)z := \langle z, y\rangle x$ for all $z \in\mathcal{H}$.
Now, following by the rank one operators in $\mathbb{B}(\mathcal{H})$, for $x, y \in \mathcal{H}$
we introduce the ``$A$-rank one operator" $x \otimes_{A} y$ as follows:
\begin{align}\label{I.5}
(x \otimes_{A} y)z = {\langle z, y\rangle}_{A}x \qquad (z \in \mathcal{H}).
\end{align}
Next, we present some properties of the $A$-rank one operators.
\begin{proposition}\label{L.4}
Let $x, y \in \mathcal{H}$ and $T\in\mathbb{B}_{A^{1/2}}(\mathcal{H})$.
Then $x \otimes_{A} y \in\mathbb{B}_{A^{1/2}}(\mathcal{H})$ and
the next assertions hold:
\begin{itemize}
\item[(i)] ${\|x \otimes_{A} y\|}_{A} = {\|x\|}_{A}{\|y\|}_{A}$
and $(x \otimes_{A} y)^{\sharp_{A}} = y \otimes_{A} x$.
\item[(ii)] $T(x \otimes_{A} y) = Tx \otimes_{A} y$ and $(x \otimes_{A} y)T = x \otimes_{A} T^{\sharp_{A}}y$.
\item[(iii)] $\widetilde{x \otimes_{A} y} = Ax \,\widetilde{\otimes}_{A}\, Ay$,
where $\widetilde{\otimes}_{A}$ is tensor product in $\mathbf{R}(A^{1/2})$.
\end{itemize}
\end{proposition}
\begin{proof}
The statements (i)-(ii) follow directly from the definition of $\otimes_{A}$.

To prove (iii), from Proposition \ref{L.2} and (i) it follows that $x \otimes_{A} y\in\mathbb{B}_{A^{1/2}}(\mathcal{H})$
and so, by Proposition \ref{L.3}, there is a unique $\widetilde{x \otimes_{A} y}\in\mathbb{B}(\mathbf{R}(A^{1/2}))$
such that $\big(\widetilde{x \otimes_{A} y}\big)M_{A} = M_{A}(x \otimes_{A} y)$.
Now, let $z\in \mathcal{H}$. By Proposition \ref{L.1}, (\ref{I.2}) and (\ref{I.5}), we have
\begin{align*}
\big(Ax \,\widetilde{\otimes}_{A}\, Ay\big)M_{A}z &= \big(Ax \,\widetilde{\otimes}_{A}\, Ay\big)Az
= [Az, Ay]Ax
\\&= {\langle z, y\rangle}_{A}Ax
= {\langle z, y\rangle}_{A}M_{A}x = M_{A}(x \otimes_{A} y)z.
\end{align*}
Thus $\big(Ax \,\widetilde{\otimes}_{A}\, Ay\big)M_{A}= M_{A}(x \otimes_{A} y)$.
Since $\widetilde{x \otimes_{A} y}$ is unique, therefore we conclude that
$\widetilde{x \otimes_{A} y} = Ax \,\widetilde{\otimes}_{A}\, Ay$.
\end{proof}
\newpage
Now we are able to establish the following result.
\begin{proposition}\label{P.5}
Let $x, y \in \mathcal{H}$. Then $w_{A}(x \otimes_{A} y) = \frac{|{\langle x, y\rangle}_{A}| + {\|x\|}_{A}{\|y\|}_{A}}{2}$.
\end{proposition}
\begin{proof}
By (\ref{I.S1.2}) we have
\begingroup\makeatletter\def\f@size{10}\check@mathfonts
\begin{align*}
w_{A}(x \otimes_{A} y) & = \displaystyle{\sup_{\theta \in \mathbb{R}}}
{\left\|\frac{e^{i\theta}x \otimes_{A} y + (e^{i\theta}x \otimes_{A} y)^{\sharp_A}}{2}\right\|}_{A}
\\& = \frac{1}{2}\displaystyle{\sup_{\theta \in \mathbb{R}}}
{\Big\|e^{i\theta}x \otimes_{A} y + e^{-i\theta}y \otimes_{A} x\Big\|}_{A}
\qquad \qquad \Big(\mbox{by Proposition \ref{L.4}(i)}\Big)
\\& = \frac{1}{2}\displaystyle{\sup_{\theta \in \mathbb{R}}}
\left(\displaystyle{\sup_{z\in\overline{\mathcal{R}(A)}, {\|z\|}_{A} = 1}}
{\Big\|\big(e^{i\theta}x \otimes_{A} y + e^{-i\theta}y \otimes_{A} x\big)z\Big\|}_{A}\right)
\\& = \frac{1}{2}\displaystyle{\sup_{\theta \in \mathbb{R}}}
\left(\displaystyle{\sup_{z\in\overline{\mathcal{R}(A)}, {\|z\|}_{A} = 1}}
{\Big\|e^{i\theta}{\langle z, y\rangle}_{A}x + e^{-i\theta}{\langle z, x\rangle}_{A}y\Big\|}_{A}\right)
\qquad \Big(\mbox{by (\ref{I.5})}\Big)
\\& = \frac{1}{2}\displaystyle{\sup_{\theta \in \mathbb{R}}}
\left(\displaystyle{\sup_{z\in\overline{\mathcal{R}(A)}, {\|Az\|}_{\mathbf{R}(A^{1/2})} = 1}}
{\Big\|A\big(e^{i\theta}{\langle z, y\rangle}_{A}x + e^{-i\theta}{\langle z, x\rangle}_{A}y\big)\Big\|}_{\mathbf{R}(A^{1/2})}\right)
\\& \qquad \qquad \qquad \qquad \qquad \qquad \qquad \qquad \quad
\qquad \qquad \qquad \Big(\mbox{by (\ref{I.3})}\Big)
\\& = \frac{1}{2}\displaystyle{\sup_{\theta \in \mathbb{R}}}
\left(\displaystyle{\sup_{z\in\overline{\mathcal{R}(A)}, {\|Az\|}_{\mathbf{R}(A^{1/2})} = 1}}
{\Big\|e^{i\theta}[Az, Ay]Ax + e^{-i\theta}[Az, Ax]Ay\Big\|}_{\mathbf{R}(A^{1/2})}\right)
\\& \qquad \qquad \qquad \qquad \qquad \qquad \qquad \qquad \quad
\qquad \qquad \qquad \Big(\mbox{by (\ref{I.2})}\Big)
\\& = \frac{1}{2}\displaystyle{\sup_{\theta \in \mathbb{R}}}
\left(\displaystyle{\sup_{z\in\overline{\mathcal{R}(A)}, {\|Az\|}_{\mathbf{R}(A^{1/2})} = 1}}
{\Big\|\big(e^{i\theta}Ax \,\widetilde{\otimes}_{A}\, Ay + (e^{i\theta}Ax \,\widetilde{\otimes}_{A}\, Ay)^*\big)Az\Big\|}_{\mathbf{R}(A^{1/2})}\right)
\\& = \displaystyle{\sup_{\theta \in \mathbb{R}}}
\left(\displaystyle{\sup_{z\in\overline{\mathcal{R}(A)}, {\|Az\|}_{\mathbf{R}(A^{1/2})} = 1}}
{\Big\|\mbox{Re}(e^{i\theta}\widetilde{x \otimes_{A} y})Az\Big\|}_{\mathbf{R}(A^{1/2})}\right)
\\& \qquad \qquad \qquad \qquad \qquad \qquad
\qquad \qquad \quad \Big(\mbox{by Proposition \ref{L.4}(iii)}\Big)
\\& = \displaystyle{\sup_{\theta \in \mathbb{R}}}{\Big\|\mbox{Re}(e^{i\theta}\widetilde{x \otimes_{A} y})\Big\|}_{\mathbf{R}(A^{1/2})}
\\& = w\big(\widetilde{x \otimes_{A} y}\big)
\qquad \qquad \qquad \qquad \qquad \qquad
\qquad \qquad \qquad \Big(\mbox{by (\ref{I.S1.3})}\Big)
\\& = \frac{|[Ax, Ay]| + {\|Ax\|}_{\mathbf{R}(A^{1/2})}{\|Ay\|}_{\mathbf{R}(A^{1/2})}}{2}
\\& \qquad \qquad \qquad \qquad \qquad \qquad \quad \Big(\mbox{by Proposition \ref{L.4}(iii) and (\ref{I.4})}\Big)
\\& = \frac{|{\langle x, y\rangle}_{A}| + {\|x\|}_{A}{\|y\|}_{A}}{2}.
\qquad \qquad \qquad \qquad \qquad \Big(\mbox{by (\ref{I.2}) and (\ref{I.3})}\Big)
\end{align*}
\endgroup
Hence $w_{A}(x \otimes_{A} y) = \frac{|{\langle x, y\rangle}_{A}| + {\|x\|}_{A}{\|y\|}_{A}}{2}$.
\end{proof}
The following lemma will be useful in the proof of the next result.
\begin{lemma}\label{L.6}
Let $T, S\in\mathbb{B}_{A^{1/2}}(\mathcal{H})$.
If $A^{1/2}Tx$ and $A^{1/2}Sx$ are linearly dependent for every $A$-unit vector $x\in \mathcal{H}$,
then there exists $\lambda\in\mathbb{C}\setminus\{0\}$ such that $A^{1/2}T = \lambda A^{1/2}S$.
\end{lemma}
\begin{proof}
Assume that $A^{1/2}T \neq \lambda A^{1/2}S$ for every nonzero complex number $\lambda$.
Then there exists an $A$-unit vector $x\in \mathcal{H}$ such that $A^{1/2}Tx \neq \lambda A^{1/2}Sx$, which contradicts our assumption.
\end{proof}
\newpage
Now we are in a position to prove our main result in this paper.
In the following, as usual, $\mathbb{T}$ is the unit circle of the complex plane,
i.e., $\mathbb{T} = \{|\alpha|:\,\, \alpha \in \mathbb{C}\}$.
\begin{theorem}\label{T.7}
Let $T, S\in\mathbb{B}_{A^{1/2}}(\mathcal{H})$. Then the following conditions are equivalent:
\begin{itemize}
\item[(i)] $A^{1/2}T = \lambda A^{1/2}S$ for some $\lambda \in \mathbb{T}$.
\item[(ii)] $w_{A}(TR) = w_{A}(SR)$ for all $R\in\mathbb{B}_{A^{1/2}}(\mathcal{H})$.
\item[(iii)] $w_{A}\big(Tx \otimes_{A} y\big) = w_{A}\big(Sx \otimes_{A} y\big)$ for all $x, y \in\mathcal{H}\setminus\{0\}$.
\end{itemize}
\end{theorem}
\begin{proof}
(i)$\Rightarrow$(ii) Let $A^{1/2}T = \lambda A^{1/2}S$ for some $\lambda \in \mathbb{T}$.
Then for every $R\in\mathbb{B}_{A^{1/2}}(\mathcal{H})$, we have
\begin{align*}
|{\langle TRx, x\rangle}_{A}| = |\langle A^{1/2}TRx, A^{1/2}x\rangle| = |\langle \lambda A^{1/2}SRx, A^{1/2}x\rangle|
= |{\langle SRx, x\rangle}_{A}|.
\end{align*}
Hence $|{\langle TRx, x\rangle}_{A}| = |{\langle SRx, x\rangle}_{A}|$.
Taking the supremum over $A$-unit vectors $x\in \mathcal{H}$, we deduce that $w_{A}(TR) = w_{A}(SR)$.

(ii)$\Rightarrow$(iii) This implication follows immediately from Proposition \ref{L.4} (ii).

(iii)$\Rightarrow$(i) Let $w_{A}\big(Tx \otimes_{A} y\big) = w_{A}\big(Sx \otimes_{A} y\big)$
for all $x, y \in\mathcal{H}\setminus\{0\}$.
Therefore, by Proposition \ref{P.5}, we have
\begin{align}\label{I.1.T.7}
|{\langle Tx, y\rangle}_{A}| + {\|Tx\|}_{A}{\|y\|}_{A}
= |{\langle Sx, y\rangle}_{A}| + {\|Sx\|}_{A}{\|y\|}_{A}.
\end{align}
Let $x\in \mathcal{H}$ with ${\|x\|}_{A} = 1$ and $y = Tx$.
So, by (\ref{I.1.T.7}) and the Cauchy–-Schwarz inequality, we obtain
\begin{align*}
2{\|Tx\|}^2_{A} \leq |{\langle Sx, Tx\rangle}_{A}| + {\|Sx\|}_{A}{\|Tx\|}_{A} \leq 2{\|Sx\|}_{A}{\|Tx\|}_{A},
\end{align*}
and hence ${\|Tx\|}_{A} \leq {\|Sx\|}_{A}$. By symmetry, therefore,
\begin{align}\label{I.2.T.7}
{\|Tx\|}_{A} = {\|Sx\|}_{A} \qquad (x\in \mathcal{H}, {\|x\|}_{A} = 1).
\end{align}
On the other hand, for any $A$-unit vector $x\in \mathcal{H}$ with $y = Sx$, by (\ref{I.1.T.7}), we get
\begin{align*}
|{\langle Tx, Sx\rangle}_{A}| + {\|Tx\|}_{A}{\|Sx\|}_{A}
= 2{\|Sx\|}^2_{A}.
\end{align*}
Therefore, by (\ref{I.2.T.7}), we infer that
$|{\langle Tx, Sx\rangle}_{A}| = {\|Tx\|}_{A}{\|Sx\|}_{A}$, or equivalently,
$|\langle A^{1/2}Tx, A^{1/2}Sx\rangle| = \|A^{1/2}Tx\|\|A^{1/2}Sx\|$.
Thus $A^{1/2}Tx$ and $A^{1/2}Sx$ are linearly dependent for every $A$-unit vector $x\in \mathcal{H}$.
So, it follows from Lemma \ref{L.6} that there exists $\lambda\in\mathbb{C}\setminus\{0\}$ such that $A^{1/2}T = \lambda A^{1/2}S$.
Finally, by (\ref{I.2.T.7}), we conclude that $|\lambda| = 1$ and the proof is completed.
\end{proof}
As an immediate consequence of Theorem \ref{T.7}, we have the following result.
\begin{corollary}
Let $T, S\in\mathbb{B}_{A^{1/2}}(\mathcal{H})$. Then the following conditions are equivalent:
\begin{itemize}
\item[(i)] $A^{1/2}T = A^{1/2}S$.
\item[(ii)] $W_{A}(TR) = W_{A}(SR)$ for all $R\in\mathbb{B}_{A^{1/2}}(\mathcal{H})$.
\item[(iii)] $W_{A}\big(Tx \otimes_{A} y\big) = W_{A}\big(Sx \otimes_{A} y\big)$ for all $x, y \in\mathcal{H}\setminus\{0\}$.
\end{itemize}
\end{corollary}
\begin{proof}
(i)$\Rightarrow$(ii) If (i) holds then, for every $R\in\mathbb{B}_{A^{1/2}}(\mathcal{H})$, we have
\begin{align*}
W_{A}(TR) &= \big\{\langle A^{1/2}TRx, A^{1/2}x\rangle: \,\, x\in \mathcal{H},\, {\|x\|}_A = 1\big\}
\\&= \big\{\langle A^{1/2}SRx, A^{1/2}x\rangle: \,\, x\in \mathcal{H},\, {\|x\|}_A = 1\big\} = W_{A}(SR).
\end{align*}

(ii)$\Rightarrow$(iii) This implication is an immediate consequence of Proposition \ref{L.4} (ii)

(iii)$\Rightarrow$(i) It follows from the equivalence (i)$\Leftrightarrow$(iii) of Theorem \ref{T.7}.
\end{proof}

A particular case of Theorem \ref{T.7} can be stated:
\begin{theorem}\label{T.8}
Let $T\in\mathbb{B}_{A^{1/2}}(\mathcal{H})$. Then the following conditions are equivalent:
\begin{itemize}
\item[(i)] $A^{1/2}T = \lambda A^{1/2}$ for some $\lambda \in \mathbb{T}$.
\item[(ii)] $w_{A}(TR) = w_{A}(R)$ for all $R\in\mathbb{B}_{A^{1/2}}(\mathcal{H})$.
\item[(iii)] $w_{A}\big(Tx \otimes_{A} y\big) = w_{A}(x \otimes_{A} y)$ for all $x, y \in\mathcal{H}\setminus\{0\}$.
\end{itemize}
\end{theorem}
On making use of Theorem \ref{T.8}, we reach the next result.
\begin{corollary}
Let $T\in\mathbb{B}_{A^{1/2}}(\mathcal{H})$. Then the following conditions are equivalent:
\begin{itemize}
\item[(i)] $A^{1/2}T = A^{1/2}$.
\item[(ii)] $W_{A}(TR) = W_{A}(R)$ for all $R\in\mathbb{B}_{A^{1/2}}(\mathcal{H})$.
\item[(iii)] $W_{A}\big(Tx \otimes_{A} y\big) = W_{A}(x \otimes_{A} y)$ for all $x, y \in\mathcal{H}\setminus\{0\}$.
\end{itemize}
\end{corollary}
\begin{remark}
By Proposition \ref{L.4}(ii) and an argument similar to that given in the proof of Theorem \ref{T.7},
for $T, S\in\mathbb{B}_{A^{1/2}}(\mathcal{H})$ it can be seen that the following conditions are equivalent:
\begin{itemize}
\item[(i)] $A^{1/2}T^{\sharp_A} = \lambda A^{1/2}S^{\sharp_A}$ for some $\lambda \in \mathbb{T}$.
\item[(ii)] $w_{A}(RT) = w_{A}(RS)$ for all $R\in\mathbb{B}_{A^{1/2}}(\mathcal{H})$.
\item[(iii)] $w_{A}\big(x \otimes_{A} T^{\sharp_A}y\big) = w_{A}\big(x \otimes_{A} S^{\sharp_A}y\big)$ for all $x, y \in\mathcal{H}\setminus\{0\}$.
\end{itemize}
\end{remark}
\textbf{Acknowledgement.}
The author would like to thank the referees for their valuable comments, which helped to improve the exposition.
\bibliographystyle{amsplain}

\end{document}